\documentclass[11pt, reqno, psamsfonts]{amsart}
\pdfoutput=1
\usepackage{amssymb}
\usepackage{amsthm}
\usepackage{amsmath}
\usepackage{latexsym}
\usepackage[T1]{fontenc}
\usepackage[utf8]{inputenc}
\usepackage[russian, french, english]{babel}
\usepackage{graphicx}
\usepackage{wrapfig}
\usepackage[justification=centering, labelfont=bf]{caption}
\usepackage{subcaption}
\usepackage{mathtools}
\usepackage[hidelinks]{hyperref}
\usepackage{amsbsy}
\usepackage{enumitem}
\usepackage{mathrsfs}
\usepackage{bbm}
\usepackage{thmtools}
\usepackage{eucal}
\usepackage[makeroom]{cancel}
\usepackage{lmodern}
\usepackage{mathabx}
\usepackage[T1]{fontenc}
\usepackage[spacing=true,kerning=true,babel=true,tracking=true]{microtype}
\usepackage[shortcuts]{extdash}
\usepackage[foot]{amsaddr} 
\usepackage[left=1in,right=1in,top=1in,bottom=1in,bindingoffset=0cm]{geometry}
\usepackage[backend=biber, style=alphabetic, sorting=nyt, maxnames=100]{biblatex}
\title{Multiplication of Weak Equivalence Classes May Be Discontinuous}
\date{}
\author{Anton~Bernshteyn}
\address{Department of Mathematics, University of Illinois at Urbana--Champaign, IL, USA and Department of Mathematical Sciences, Carnegie Mellon University, Pittsburgh, PA, USA}
\email{bernsht2@illinois.edu; abernsht@math.cmu.edu}
\thanks{This research is supported in part by the Waldemar J., Barbara G., and Juliette Alexandra Trjitzinsky Fellowship.}

\newtheoremstyle{bfnote}%
{}{}%
{\slshape}{}%
{\bfseries}{\bfseries.}%
{ }%
{\thmname{#1}\thmnumber{ #2}\thmnote{ \ep{\normalfont{}#3}}}

\newtheoremstyle{defbfnote}%
{}{}%
{}{}%
{\bfseries}{.}%
{ }%
{\thmname{#1}\thmnumber{ #2}\thmnote{ (#3)}}

\newtheoremstyle{claim}%
{}{}%
{\slshape}{}%
{\itshape}{.}%
{ }%
{\thmname{#1}\thmnumber{ #2}\thmnote{ \ep{\normalfont{}#3}}}

\theoremstyle{bfnote}
\newtheorem{theo}{Theorem}[section]
\newtheorem{prop}[theo]{Proposition}
\newtheorem{lemma}[theo]{Lemma}
\newtheorem{corl}[theo]{Corollary}

\newtheorem{lemdef}[theo]{Lemma/Definition}
\newtheorem*{claim*}{Claim}
\newtheorem{smallclaim}{Claim}[theo]

\newcommand*{\myproofname}{Proof}
\newenvironment{claimproof}[1][\myproofname]{\begin{proof}[#1]}{\end{proof}}

\theoremstyle{definition}
\newtheorem{defn}[theo]{Definition}

\newtheorem{ques}[theo]{Question}

\newtheorem*{exmp*}{Example}

\theoremstyle{remark}
\newtheorem*{ques*}{Question}
\newtheorem*{remk*}{Remark}

\newcommand{\0}{\varnothing}
\newcommand{\set}[1]{\{#1\}}
\newcommand{\proj}{\mathrm{proj}}

\newcommand{\acts}{\curvearrowright}
\newcommand{\N}{\mathbb{N}}
\newcommand{\Z}{\mathbb{Z}}
\newcommand{\F}{\mathbb{F}}
\newcommand{\R}{\mathbb{R}}
\newcommand{\Q}{\mathbb{Q}}
\renewcommand{\c}{^{\mathsf{c}}}

\renewcommand{\epsilon}{\varepsilon}
\renewcommand{\phi}{\varphi}
\renewcommand{\theta}{\vartheta}
\renewcommand{\leq}{\leqslant}
\renewcommand{\geq}{\geqslant}
\renewcommand{\preceq}{\preccurlyeq}

\newcommand{\symdif}{\bigtriangleup}

\newcommand{\fins}[1]{[#1]^{<\infty}}

\renewcommand{\G}{\Gamma}

\newcommand{\defeq}{\coloneqq}

\newcommand{\emphd}[1]{{\fontseries{b}\selectfont\textsf{#1}}}

\newcommand{\Ball}{\mathrm{Ball}}
\newcommand{\wec}[1]{{[#1]}}
\newcommand{\wecSet}[3]{{\theta_{#2,#3}(#1)}}
\newcommand{\wecSetFun}[4]{{\theta_{#2,#3}(#1, #4)}}
\newcommand{\Meas}{{\mathrm{Meas}}}
\newcommand{\Step}{{\mathrm{Step}}}
\newcommand{\K}{{\mathcal{K}}}
\newcommand{\W}{{\mathcal{W}}}
\newcommand{\WFree}{\mathcal{FW}}
\newcommand{\dist}{{\mathrm{dist}}}
\newcommand{\bemph}[1]{{\normalfont#1}} 
\newcommand{\ep}[1]{\bemph{(}#1\bemph{)}} 

\newcommand{\SL}{\mathrm{SL}}
\newcommand{\uhalf}{u}

\usepackage{xspace}
\newcommand{\pmp}{{p.m.p.}\xspace}

\numberwithin{equation}{section}

\usepackage{etoolbox}

\patchcmd{\subsection}{\normalfont}{\itshape\bfseries}{}{}

\makeatletter
\def\@seccntformat#1{%
	\protect\textup{%
		\protect\@secnumfont
		\expandafter\protect\csname format#1\endcsname 
		\csname the#1\endcsname
		\protect\@secnumpunct
	}%
}

\makeatother

\makeatletter
\newcommand{\neutralize}[1]{\expandafter\let\csname c@#1\endcsname\count@}
\makeatother

\renewcommand{\mkbibnamefirst}[1]{\textsc{#1}}
\renewcommand{\mkbibnamelast}[1]{\textsc{#1}}
\renewcommand{\mkbibnameprefix}[1]{\textsc{#1}}
\renewcommand{\mkbibnameaffix}[1]{\textsc{#1}}
\renewcommand{\finentrypunct}{}

\renewbibmacro{in:}{}

\renewbibmacro*{volume+number+eid}{%
	\printfield{volume}%
	\setunit*{\addnbspace}
	\printfield{number}%
	\setunit{\addcomma\space}%
	\printfield{eid}}

\DeclareFieldFormat[article]{volume}{\textbf{#1}\space}
\DeclareFieldFormat[article]{number}{\mkbibparens{#1}}

\DeclareFieldFormat{journaltitle}{#1,}
\DeclareFieldFormat[thesis]{title}{\mkbibemph{#1}\addperiod}
\DeclareFieldFormat[article, unpublished, thesis]{title}{\mkbibemph{#1},}
\DeclareFieldFormat[book]{title}{\mkbibemph{#1}\addperiod}
\DeclareFieldFormat[unpublished]{howpublished}{#1, }

\DeclareFieldFormat{pages}{#1}

\DeclareFieldFormat[article]{series}{Ser.~#1\addcomma}

\renewcommand*{\bibfont}{\small}

\begin{document}
	\pagestyle{plain}
	
	\maketitle
	
	\begin{abstract}
		For a countably infinite group $\G$, let $\W_\G$ denote the space of all weak equivalence classes of measure\-/preserving actions of $\G$ on atomless standard probability spaces, equipped with the compact metrizable topology introduced by Ab\'ert and Elek. There is a natural multiplication operation on $\W_\G$ (induced by taking products of actions) that makes $\W_\G$ an Abelian semigroup. Burton, Kechris, and Tamuz showed that if $\G$ is amenable, then $\W_\G$ is a topological semigroup, i.e., the product map $\W_\G \times \W_\G \to \W_\G \colon (\mathfrak{a}, \mathfrak{b}) \mapsto \mathfrak{a} \times \mathfrak{b}$ is continuous. In contrast to that, we prove that if $\G$ is a Zariski dense subgroup of $\SL_d(\Z)$ for some $d \geq 2$ (for instance, if $\G$ is a non-Abelian free group), then multiplication on $\W_\G$ is discontinuous, even when restricted to the subspace $\WFree_\G$ of all free weak equivalence classes.
	\end{abstract}
	
	\section{Introduction}
	
	\noindent Throughout, $\G$ denotes a countably infinite group with identity element $\mathbf{1}$. In this paper, we study \emphd{probability measure\-/preserving \ep{\pmp} actions} of $\G$, i.e., actions of the form $\alpha \colon \G \acts (X, \mu)$, where $(X, \mu)$ is a standard probability space and $\mu$ is preserved by $\alpha$. We say that a \pmp action $\alpha \colon \G \acts (X, \mu)$ is \emphd{free} if the stabilizer of $\mu$-almost every point $x \in X$ is trivial, i.e., if
	\[
		\mu(\set{x \in X \,:\, \gamma \cdot x \neq x \text{ for all } \mathbf{1} \neq \gamma \in \G}) = 1.
	\]

	 The concepts of \emph{weak containment} and \emph{weak equivalence} of \pmp actions of $\G$ were introduced by Kechris in~\cite[Section 10(C)]{K_book}. They were inspired by the analogous notions for unitary representations and are closely related to the so-called \emph{local\-/global convergence} in the theory of graph limits~\cite{LocalGlobal}. Roughly speaking, a \pmp action $\alpha \colon \G \acts (X, \mu)$ is weakly contained in another \pmp action $\beta \colon \G \acts (Y, \nu)$, in symbols $\alpha \preceq \beta$, if the interaction between any finite measurable partition of $X$ and a finite collection of elements of $\G$ can be simulated, with arbitrarily small error, by a measurable partition of $Y$ (see \S\ref{subsec:weak_cont_defn} for the precise definition). If both $\alpha \preceq \beta$ and $\beta \preceq \alpha$, then $\alpha$ and $\beta$ are said to be weakly equivalent, in symbols $\alpha \simeq \beta$. 
	
	The relation of weak equivalence is much coarser than the conjugacy relation, which makes it relatively well-behaved. On the other hand, several interesting parameters associated with \pmp actions---such as their cost, type, etc.---turn out to be invariants of weak equivalence. Due to these favorable properties, the relations of weak containment and weak equivalence have attracted a considerable amount of attention in recent years. For a survey of the topic, see \cite{BK}.
	
	We denote the weak equivalence class of a \pmp action $\alpha$ by $\wec{\alpha}$. Define
	\[
		\W_\G \defeq \set{\wec{\alpha} \,:\, \alpha \colon \G \acts (X, \mu), \text{ where } (X, \mu) \text{ is atomless}}.
	\]
	A weak equivalence class $\mathfrak{a} \in \W_\G$ is \emphd{free} if $\mathfrak{a} = \wec{\alpha}$ for some free \pmp action $\alpha$. Let
	\[
		\WFree_\G \defeq \set{\mathfrak{a} \in \W_\G \,:\, \mathfrak{a} \text{ is free}}.
	\]
	Freeness is an invariant of weak equivalence \cite[Theorem~3.4]{BK}; in other words, if $\mathfrak{a} \in \WFree_\G$, then all \pmp actions $\alpha$ with $\wec{\alpha} = \mathfrak{a}$ are free.
	
	Ab\'ert and Elek introduced a natural topology on $\W_\G$ and proved that it is compact and metrizable \cite[Theorem~1]{AbertElek} (see also \cite[Theorem~10.1]{BK}). Under this topology, $\WFree_\G$ becomes a closed subset of $\W_\G$ \cite[Corollary~10.7]{BK}. We review the definition of this topology in \S\ref{subsec:top_defn}.
	
	The \emphd{product} of two \pmp actions $\alpha \colon \G \acts (X, \mu)$ and $\beta \colon \G \acts (Y, \nu)$ is the action
	\[
		\alpha \times \beta \colon \G \acts (X \times Y, \mu \times \nu), \qquad \text{given by} \qquad \gamma \cdot (x, y) \defeq (\gamma \cdot x, \gamma \cdot y).
	\]
	It can be easily seen that the weak equivalence class of $\alpha \times \beta$ is determined by the weak equivalence classes of $\alpha$ and $\beta$ (for completeness, we include a proof of this fact---see Corollary~\ref{corl:mult}), hence there is a well-defined multiplication operation on $\W_\G$, namely
	\[
	\wec{\alpha} \times \wec{\beta} \defeq \wec{\alpha \times \beta}.
	\]
	Equipped with this operation, $\W_\G$ is an Abelian semigroup and $\WFree_\G$ is a subsemigroup (in fact, an ideal) in $\W_\G$. We are interested in the following natural question:
	
	\begin{ques}[{\cite[Problem 10.36]{BK}}]\label{ques:main}
		Is $\W_\G$ a topological semigroup? In other words, is the map $\W_\G \times \W_\G \to \W_\G \colon (\mathfrak{a}, \mathfrak{b}) \mapsto \mathfrak{a} \times \mathfrak{b}$ continuous?
	\end{ques}
	
	Burton, Kechris, and Tamuz answered Question~\ref{ques:main} positively when the group $\G$ is amenable \cite[Theorem~10.37]{BK}. A crucial role in their argument is played by the identification of the space $\W_\G$ for amenable $\G$ with the space of the so-called \emph{invariant random subgroups} of $\G$ \cite[Theorem~10.6]{BK}. Note that the continuity of multiplication on the subspace $\WFree_\G$ for amenable $\G$ is a triviality, since if $\G$ is amenable, then $\WFree_\G$ contains only a single point \cite[15]{BK}. On the other hand, if $\G$ is nonamenable, then $\WFree_\G$ has cardinality continuum \cite[Remark 4.3]{T-D}.
	
	The goal of this paper is to give a \emph{negative} answer to Question~\ref{ques:main} for a certain class of nonamenable groups $\G$, including the non-Abelian free groups:
	
	\begin{theo}\label{theo:main}
		Let $d \geq 2$ and let $\G \leq \SL_d(\Z)$ be a subgroup that is Zariski dense in $\SL_d(\R)$.
		
		\begin{enumerate}[label={\ep{\normalfont{}\arabic*}}]
			\item\label{item:squaring} The map
			$
			\WFree_\G \to \WFree_\G \colon \mathfrak{a} \mapsto \mathfrak{a} \times \mathfrak{a}
			$
			is discontinuous.
			
			\item\label{item:fiber} There is $\mathfrak{b} \in \WFree_\G$ such that the map
			$
			\WFree_\G \to \WFree_\G \colon \mathfrak{a} \mapsto \mathfrak{a} \times \mathfrak{b}
			$
			is discontinuous.
		\end{enumerate}
	\end{theo}
	
	As observed in \cite[\S10.2]{BK}, part \ref{item:squaring} of Theorem~\ref{theo:main} yields the following corollary:
	
	\begin{corl}
		There exists a countable group $\Delta$ with a normal subgroup $\G \lhd \Delta$ of index $2$ such that the co\-/induction map $\W_\G \to \W_\Delta$ is discontinuous.
	\end{corl}
	\begin{proof}[\textsc{Proof}]
		Let $d \geq 2$ and let $\G \leq \SL_d(\Z)$ be any Zariski dense subgroup. Set $\Delta \defeq \G \times (\Z/2\Z)$ and identify $\G$ with a normal subgroup of $\Delta$ of index $2$ in the obvious way. Then, for any \pmp action $\alpha$ of $\G$, the restriction of the co-induced action $\mathrm{CInd}_\G^\Delta(\alpha)$ back to $\G$ is isomorphic to $\alpha \times \alpha$. Since the restriction map $\W_\Delta \to \W_\G$ is continuous \cite[Proposition 10.10]{BK}, Theorem~\ref{theo:main}\ref{item:squaring} forces the co-induction map $\W_\G \to \W_\Delta$ to be discontinuous. For details, see \cite[\S10.2]{BK}.
	\end{proof}
	
	In view of Theorem~\ref{theo:main} and the result of Burton, Kechris, and Tamuz, it is tempting to conjecture that $\W_\G$ is a topological semigroup \emph{if and only if} $\G$ is amenable. However, at this point we do not even know whether multiplication of weak equivalence classes is discontinuous for every countable group that contains a non-Abelian free subgroup.
	
	Our proof of Theorem~\ref{theo:main} provides explicit examples of sequences of \pmp actions that witness the discontinuity of multiplication on $\W_\G$. We describe one such example here. Let $d \geq 2$ and let $\G \leq \SL_d(\Z)$ be a Zariski dense subgroup. For a prime $p$, let $\Z_p$ denote the ring of $p$-adic integers. Then $\SL_d(\Z_p)$ is an infinite profinite group. Since $\SL_d(\Z)$ naturally embeds in $\SL_d(\Z_p)$, we may identify $\G$ with a subgroup of $\SL_d(\Z_p)$ and consider the left multiplication action $\alpha_p \colon \G \acts \SL_d(\Z_p)$, which we view as a \pmp action by putting the Haar probability measure on $\SL_d(\Z_p)$. Let $\mathfrak{a}_p$ denote the weak equivalence class of $\alpha_p$. Using the compactness of $\W_\G$, we can pick an increasing sequence of primes $p_0$, $p_1$, \ldots{} such that the sequence $(\mathfrak{a}_{p_i})_{i\in\N}$ converges in $\W_\G$ to some weak equivalence class $\mathfrak{a}$. Then it follows from our results that the sequence $(\mathfrak{a}_{p_i} \times \mathfrak{a}_{p_i})_{i \in \N}$ does \emph{not} converge to $\mathfrak{a} \times \mathfrak{a}$, thus demonstrating that multiplication on $\W_\G$ is discontinuous.
	 	
	The main tools that we use to prove Theorem~\ref{theo:main} come from the study of expansion properties in finite groups of Lie type, specifically the groups $\SL_d(\Z/n\Z)$ for $n \in \N^+$. Our primary reference for this subject is the book \cite{Tao_book}.
	
	This paper is organized as follows. Section~\ref{sec:prelim} contains some basic definitions (such as the definition of the weak equivalence relation and the topology on the space $\W_\G$) and a few preliminary results. In Section~\ref{sec:step}, we introduce the terminology pertaining to step functions and use it in Section~\ref{sec:criterion} to prove Theorem~\ref{theo:criterion}, an explicit criterion for continuity of multiplication, which is of some independent interest. Finally, the proof of Theorem~\ref{theo:main} is presented in Section~\ref{sec:proof}. 
	
	\subsection*{Acknowledgements}
	
	I am grateful to Alexander Kechris for bringing Question~\ref{ques:main} to my attention, to Anush Tserunyan for insightful discussions, and to the anonymous referee for carefully reading the manuscript and providing helpful positive comments.
	
	\section{Preliminaries}\label{sec:prelim}
	
	\subsection{Basic notation, conventions, and terminology}
	
	We use $\N$ to denote the set of all nonnegative integers and let $\N^+ \defeq \N \setminus \set{0}$. Each $k \in \N$ is identified with the set $\set{i \in \N \,:\, i < k}$.
	
	For a set $S$, we use $\fins{S}$ to denote the set of all finite subsets of $S$.
	
	For a standard probability space $(X, \mu)$ and $k \in \N^+$, we use $\Meas_k(X,\mu)$ to denote the space of all measurable maps $f \colon X \to k$, equipped with the pseudometric
	\[
	\dist_\mu(f, g) \defeq \mu(\set{x \in X \,:\, f(x) \neq g(x)}).
	\]
	
	\subsection{The relations of weak containment and weak equivalence}\label{subsec:weak_cont_defn}
	
	A number of equivalent definitions of weak containment exist, and several of them can be found in~\cite[\S\S2.1, 2.2]{BK}. We use the definition given in \cite[\S2.2(1)]{BK}, as it is particularly well-suited for introducing the topology on the space of weak equivalence classes.
	
	Let $\alpha \colon \G \acts (X,\mu)$ be a \pmp action of $\G$. For $S \in \fins{\G}$, $k \in \N^+$, and $f \in \Meas_k(X,\mu)$, define
	\[
	\wecSetFun \alpha S k f \colon S \times k \times k \to [0;1]
	\]
	by setting, for all $\gamma \in S$ and $i$, $j < k$,
	\[
	\wecSetFun \alpha S k f (\gamma, i, j) \defeq \mu(\set{x \in X \,:\, f(x) = i,\ f(\gamma \cdot x) = j}).
	\]
	Thus, $\wecSetFun \alpha S k f$ is a vector in the unit cube $\mathbb{I}_{S,k} \defeq [0;1]^{S \times k \times k}$. For $F \subseteq \Meas_k(X, \mu)$, let
	\[
	\wecSetFun \alpha S k F \defeq \set{\wecSetFun \alpha S k f \,:\, f \in F},
	\]
	and define $\wecSet \alpha S k $ to be the closure of the set $\wecSetFun \alpha S k {\Meas_k(X, \mu)}$ in $\mathbb{I}_{S,k}$.
	
	\begin{defn}\label{defn:weak}
		Let $\alpha$ and $\beta$ be \pmp actions of $\G$. We say that $\alpha$ is \emphd{weakly contained} in $\beta$, in symbols $\alpha \preceq \beta$, if for all $S \in \fins{\G}$ and $k \in \N^+$, we have \[\wecSet \alpha S k \subseteq \wecSet \beta S k.\] If simultaneously $\alpha \preceq \beta$ and $\beta \preceq \alpha$, i.e., if for all $S \in \fins{\G}$ and $k \in \N^+$, we have \[\wecSet \alpha S k = \wecSet \beta S k,\] then $\alpha$ and $\beta$ are said to be \emphd{weakly equivalent}, in symbols $\alpha \simeq \beta$.
	\end{defn}
	
	In view of Definition~\ref{defn:weak}, we refer to the sequence
	\[
	\wec{\alpha} \defeq (\wecSet \alpha S k )_{S,k},
	\]
	where $S$ and $k$ run over $\fins{\G}$ and $\N^+$ respectively, as the \emphd{weak equivalence class} of $\alpha$. Let
	\[
	\W_\G \defeq \set{\wec{\alpha} \,:\, \alpha \colon \G \acts (X, \mu), \text{ where } (X, \mu) \text{ is atomless}}.
	\]
	A weak equivalence class $\mathfrak{a} \in \W_\G$ is \emphd{free} if $\mathfrak{a} = \wec{\alpha}$ for some free \pmp action $\alpha$. Define
	\[
	\WFree_\G \defeq \set{\mathfrak{a} \in \W_\G \,:\, \mathfrak{a} \text{ is free}}.
	\]
	
	\begin{theo}[{\cite[Theorem~3.4]{BK}}]
		Let $\alpha$ and $\beta$ be \pmp actions of~$\G$. If $\alpha$ is free and $\alpha \preceq \beta$, then $\beta$ is also free. In particular, if $\mathfrak{a} \in \WFree_\G$, then all \pmp actions $\alpha$ with $\wec{\alpha} = \mathfrak{a}$ are free.
	\end{theo}
	
	\subsection{The space of weak equivalence classes}\label{subsec:top_defn}
	
	We now proceed to define the topology on $\W_\G$. For $S \in \fins{\G}$ and $k \in \N^+$, the cube $\mathbb{I}_{S, k} = [0;1]^{S \times k \times k}$ is equipped with the \emphd{$\infty$-metric}:
	\[
	\dist_\infty(u, v) = \|u - v\|_\infty \defeq \max_{\gamma, i, j} |u(\gamma, i, j) - v(\gamma,i,j)|.
	\]
	Let $\K(\mathbb{I}_{S, k})$ denote the set of all nonempty compact subsets of $\mathbb{I}_{S, k}$. For $C \in \K(\mathbb{I}_{S, k})$, let
	\[
	\Ball_\epsilon(C) \defeq \set{u \in \mathbb{I}_{S, k} \,:\, \dist_\infty(u, C) < \epsilon}.
	\]
	Define the \emphd{Hausdorff metric} on $\K(\mathbb{I}_{S, k})$ by
	\[
	\dist_H(C_1, C_2) \defeq \inf \set{\epsilon > 0 \,:\, C_1 \subseteq \Ball_\epsilon(C_2) \text{ and } C_2 \subseteq \Ball_\epsilon(C_1)}.
	\]
	This metric makes $\K(\mathbb{I}_{S, k})$ into a compact space \cite[Theorem 4.26]{K_DST}. By definition, for any \pmp action $\alpha$, we have $\wecSet {\alpha} S k \in \K(\mathbb{I}_{S, k})$, so $\W_\G$ is a subset of the compact metrizable space \[\prod_{S, k} \K(\mathbb{I}_{S, k}),\] where the product is over all $S \in \fins{\G}$ and $k \in \N^+$, and as such, $\W_\G$ inherits a relative topology.
	
	The following fundamental result is due to Ab\'ert and Elek:
	
	\begin{theo}[{Ab\'ert--Elek \cite[Theorem~1]{AbertElek}; see also \cite[Theorem~10.1]{BK}}]
		The set $\W_\G$ is closed in $\prod_{S, k} \K(\mathbb{I}_{S,k})$. In other words, the space $\W_\G$ is compact.
	\end{theo}
	
	The subspace $\WFree_\G$ is also compact:
	
	\begin{theo}[{\cite[Corollary~10.7]{BK}}]
		The set $\WFree_\G$ is closed in $\W_\G$.
	\end{theo}
	
	\subsection{The map $\wecSetFun \alpha S k -$ is Lipschitz}
	
	The purpose of this short subsection is to record the following simple observation:
	
	\begin{prop}\label{prop:Lipschitz}
		Let $\alpha \colon \G \acts (X, \mu)$ be a \pmp action of $\G$. If $k \in \N^+$ and  $f$, $g \in \Meas_k(X,\mu)$, then, for any $S \in \fins{\G}$,
		\[
		\dist_\infty(\wecSetFun \alpha S k f, \wecSetFun \alpha S k {g}) \,\leq\, 2 \cdot \dist_\mu(f, g).
		\]
	\end{prop}
	\begin{proof}[\textsc{Proof}]
		Take any $\gamma \in S$ and $i$, $j < k$ and let
		\[
			A \defeq \set{x \in X \,:\, f(x) = i, \ f(\gamma \cdot x) = j} \qquad \text{and} \qquad B \defeq \set{x \in X \,:\, g(x) = i, \ g(\gamma \cdot x) = j}.
		\]
		Then, by definition,
		\[
		|\wecSetFun \alpha S k f (\gamma, i, j) \,-\, \wecSetFun \alpha S k {g} (\gamma, i, j) | \,=\, |\mu(A) \,-\, \mu(B)| \,\leq\, \mu(A \symdif B).
		\]
		If $x \in A \symdif B$, then $f(x) \neq g(x)$ or $f(\gamma \cdot x) \neq g(\gamma \cdot x)$, so $\mu(A \symdif B) \leq 2 \cdot \dist_\mu(f,g)$, as desired.
	\end{proof}
	
	\section{Step functions}\label{sec:step}
	
	\noindent In this section we establish some basic facts pertaining to step functions on products of probability spaces. In particular, we show that multiplication is a well-defined operation on $\W_\G$.
	
	To begin with, we need a few definitions. Let $(X, \mu)$ and $(Y, \nu)$ be standard probability spaces and let $k$, $N \in \N^+$. We call a map $f \in \Meas_k(X \times Y, \mu \times \nu)$ an \emphd{$N$-step function} if there exist
	\[
	g \in \Meas_N(X, \mu), \qquad h \in \Meas_N(Y, \nu), \qquad \text{and} \qquad \phi \colon N \times N \to k,
	\]
	such that $f = \phi \circ (g, h)$, i.e., we have
	\[
	f(x, y) = \phi(g(x), h(y)) \qquad \text{for all } x \in X \text{ and } y \in Y.
	\]
	Let $\Step_{k,N}(X, \mu; Y, \nu) \subseteq \Meas_k(X \times Y, \mu \times \nu)$ denote the set of all $N$-step functions and let
	\begin{equation}\label{eq:step}
	\Step_k(X, \mu; Y, \nu) \defeq \bigcup_{N \in \N^+} \Step_{k,N}(X, \mu; Y, \nu).
	\end{equation}
	The maps in $\Step_k(X, \mu; Y, \nu)$ are called \emphd{step functions}. Note that the union in~\eqref{eq:step} is increasing. It is a basic fact in measure theory that the set $\Step_k(X, \mu; Y, \nu)$ is dense in $\Meas_k(X \times Y, \mu \times \nu)$.

	It will be useful to have a concrete description of the vectors of the form $\wecSetFun {\alpha \times \beta} S k f$, where $f$ is a step function. To that end, we introduce the following operation:
	
	\begin{defn}
		Let $k$, $N \in \N^+$ and $\phi \colon N \times N \to k$. Given $S \in \fins{\G}$ and vectors $u$, $v \in \R^{S\times N \times N}$, the \emphd{$\phi$-convolution} $u \ast_\phi v \in \R^{S \times k \times k}$ of $u$ and $v$ is given by the formula
	\[
	(u \ast_\phi v)(\gamma, i, j) \defeq \sum_{(a, b) \in \phi^{-1}(i)}\, \sum_{(c, d) \in \phi^{-1}(j)} \, u(\gamma, a, c) \cdot v(\gamma, b, d).
	\]
	\end{defn}

	The next proposition is an immediate consequence of the definitions:
	
	\begin{prop}\label{prop:conv}
		Let $\alpha \colon \G \acts (X, \mu)$, $\beta \colon \G \acts (Y, \nu)$ be \pmp actions of~$\G$. Let $k$, $N \in \N^+$ and
		\[
		g \in \Meas_N(X, \mu), \qquad h \in \Meas_N(Y, \nu), \qquad \text{and} \qquad \phi \colon N \times N \to k.
		\]
		Set $f \coloneqq \phi \circ (g, h)$. Then, for any $S \in \fins{\G}$,
		\[
		\wecSetFun {\alpha \times \beta} S k f \,=\, \wecSetFun \alpha S N g \ast_\phi \wecSetFun \beta S N h.
		\]
	\end{prop}
	
	It is useful to note that the $\phi$-convolution operation is Lipschitz on $\mathbb{I}_{S,N}$:
	
	\begin{prop}\label{prop:conv_Lip}
		Let $k$, $N \in \N^+$ and $\phi \colon N \times N \to k$. For all $S \in \fins{\G}$ and $u$, $v$, $\tilde{u}$, $\tilde{v} \in \mathbb{I}_{S,N}$,
		\[
		\dist_\infty (u \ast_\phi v, \, \tilde{u} \ast_\phi \tilde{v}) \,\leq\, N^4 \cdot (\dist_\infty(u, \tilde{u}) + \dist_\infty (v, \tilde{v})).
		\]
	\end{prop}
	\begin{proof}[\textsc{Proof}]
		Since
		\[
		\dist_\infty (u \ast_\phi v, \, \tilde{u} \ast_\phi \tilde{v}) \,\leq\, \dist_\infty (u \ast_\phi v, \, \tilde{u} \ast_\phi v) + \dist_\infty (\tilde{u} \ast_\phi v, \, \tilde{u} \ast_\phi \tilde{v}),
		\]
		it suffices to prove the inequality when, say, $v = \tilde{v}$. To that end, take $\gamma \in S$ and $i$, $j < k$. We have
		\begin{align*}
		|(u \ast_\phi v)(\gamma, i, j) \,-\, (\tilde{u} \ast_\phi v)&(\gamma, i, j)| \,\leq\, \sum_{(a, b) \in \phi^{-1}(i)}\, \sum_{(c, d) \in \phi^{-1}(j)} \, |u(\gamma, a, c) \,-\, \tilde{u}(\gamma, a, c)| \cdot v(\gamma, b, d)\\
		&\leq\, |\phi^{-1}(i)|\cdot |\phi^{-1}(j)| \cdot \dist_\infty(u, \tilde{u}) \,\leq\, N^4 \cdot \dist_\infty(u, \tilde{u}). \qedhere
		\end{align*}
	\end{proof}
	
	\begin{corl}\label{corl:mult}
		If $\alpha$, $\tilde{\alpha}$, and $\beta$ are \pmp actions of $\G$ and $\alpha \preceq \tilde{\alpha}$, then $\alpha \times \beta \preceq \tilde{\alpha} \times \beta$. In particular, the multiplication operation on $\W_\G$ is well-defined.
	\end{corl}
	\begin{proof}[\textsc{Proof}]
		Let $\alpha \colon \G \acts (X, \mu)$, $\tilde{\alpha} \colon \G \acts (\tilde{X}, \tilde{\mu})$, and $\beta \colon \G \acts (Y, \nu)$ be \pmp actions of $\G$ and suppose that $\alpha \preceq \tilde{\alpha}$. Take any $S \in \fins{\G}$ and $k \in \N^+$. By Proposition~\ref{prop:Lipschitz} and since $\Step_k(X, \mu; Y, \nu)$ is dense in $\Meas_k(X \times Y, \mu \times \nu)$, it suffices to show that for all $f \in \Step_{k}(X, \mu; Y, \nu)$ and $\epsilon > 0$, there is $\tilde{f} \in \Meas_k(\tilde{X} \times Y, \tilde{\mu} \times \nu)$ such that
		\[
		\dist_\infty(\wecSetFun {\alpha \times \beta} S k f, \, \wecSetFun {\tilde{\alpha} \times \beta} S k {\tilde{f}}) \,< \,\epsilon.
		\]
		Let $N \in \N^+$, $g \in \Meas_N(X, \mu)$, $h \in \Meas_N(Y, \nu)$, and $\phi \colon N \times N \to k$ be such that $f = \phi \circ (g, h)$. Since $\alpha \preceq \tilde{\alpha}$, there is a map $\tilde{g} \in \Meas_N(\tilde{X}, \tilde{\mu})$ with
		\[
		\dist_\infty(\wecSetFun \alpha S N g,\, \wecSetFun {\tilde{\alpha}} S N {\tilde{g}}) \,<\, \epsilon N^{-4}.
		\]
		From Propositions~\ref{prop:conv} and \ref{prop:conv_Lip}, it follows that the map $\tilde{f} \defeq \phi \circ (\tilde{g}, h)$ is as desired.
	\end{proof}

	\section{A criterion of continuity}\label{sec:criterion}
	
	\noindent The purpose of this section is to establish an explicit necessary and sufficient condition for the continuity of multiplication on  the space of weak equivalence classes.
	
	Recall that a subset $Y$ of a metric space $X$ is called an \emphd{$\epsilon$-net} if for every $x \in X$, there is $y \in Y$ such that the distance between $x$ and $y$ is less than $\epsilon$. 
	
	\begin{lemdef}
		Let $\alpha \colon \G \acts (X, \mu)$ and $\beta \colon \G \acts (Y, \nu)$ be \pmp actions of~$\G$. For any $S \in \fins{\G}$, $k \in \N^+$, and $\epsilon > 0$, there exists $N \in \N^+$ such that the set
		\[
			\wecSetFun {\alpha \times \beta} S k {\Step_{k,N}(X, \mu; Y, \nu)}
		\]
		is an $\epsilon$-net in $\wecSet {\alpha \times \beta} S k$. We denote the smallest such $N$ by $N_{S,k}(\alpha, \beta, \epsilon)$.
		
		Furthermore, the value $N_{S, k}(\alpha, \beta, \epsilon)$ is determined by the weak equivalence classes of $\alpha$ and $\beta$, so we can define $N_{S, k}(\wec{\alpha}, \wec{\beta}, \epsilon) \defeq N_{S,k}(\alpha, \beta, \epsilon)$.
	\end{lemdef}
	\begin{proof}[\textsc{Proof}]
		By Proposition~\ref{prop:Lipschitz} and since $\Step_k(X, \mu; Y, \nu)$ is dense in $\Meas_k(X \times Y, \mu \times \nu)$, the set
		\[
			\wecSetFun {\alpha \times \beta} S k {\Step_{k}(X, \mu; Y, \nu)}
		\]
		is dense in $\wecSet {\alpha \times \beta} S k$. The existence of $N_{S,k}(\alpha, \beta, \epsilon)$ then follows since $\wecSet {\alpha \times \beta} S k$ is compact.
		
		To prove the ``furthermore'' part, let $\tilde{\alpha} \colon \G \acts (\tilde{X}, \tilde{\mu})$ and $\tilde{\beta} \colon \G \acts (\tilde{Y}, \tilde{\nu})$ be \pmp actions of $\G$ such that $\alpha \simeq \tilde{\alpha}$ and $\beta \simeq \tilde{\beta}$. Set $N \defeq N_{S,k}(\alpha, \beta, \epsilon)$. We have to show that
		\[
			N_{S,k}(\tilde{\alpha}, \tilde{\beta}, \epsilon) \leq N.
		\]
		Take any $u \in \wecSet {\tilde{\alpha} \times \tilde{\beta}} S k$. Corollary~\ref{corl:mult} implies that $\wecSet {\tilde{\alpha} \times \tilde{\beta}} S k = \wecSet {\alpha \times \beta} S k$, so, by the choice of $N$, there is $f \in \Step_{k, N}(X, \mu; Y, \nu)$ such that
		\[
			\delta \defeq \epsilon - \dist_\infty(u,\wecSetFun {\alpha \times \beta} S k f) \,>\, 0.
		\]
		Let $g \in \Meas_N(X, \mu)$, $h \in \Meas_N(Y, \nu)$, and $\phi \colon N \times N \to k$ be such that $f = \phi \circ (g, h)$. Since we have $\alpha \simeq \tilde{\alpha}$ and $\beta \simeq \tilde{\beta}$, there exist maps $\tilde{g} \in \Meas_N(\tilde{X}, \tilde{\mu})$ and $\tilde{h} \in \Meas_N(\tilde{Y}, \tilde{\nu})$ with
		\[
		\dist_\infty(\wecSetFun {\tilde{\alpha}} S N {\tilde{g}},\, \wecSetFun \alpha S N g) + \dist_\infty(\wecSetFun {\tilde{\beta}} S N {\tilde{h}},\, \wecSetFun \beta S N h) \,<\, \delta N^{-4}.
		\]
		Set $\tilde{f} \defeq \phi \circ (\tilde{g}, \tilde{h})$. Then $\tilde{f} \in \Step_{k, N} (\tilde{X}, \tilde{\mu}; \tilde{Y}, \tilde{\nu})$, and, from Propositions~\ref{prop:conv} and \ref{prop:conv_Lip}, it follows that
		\[
			\dist_\infty(u, \wecSetFun {\tilde{\alpha} \times \tilde{\beta}} S k {\tilde{f}}) \,<\, \epsilon.
		\]
		Since $u$ was chosen arbitrarily, this concludes the proof.
	\end{proof}

	Now we can state the main result of this section:
	
	\begin{theo}\label{theo:criterion}
		Let $\mathcal{C} \subseteq \W_\G \times \W_\G$ be a closed set. The following statements are equivalent:
		\begin{enumerate}[label={\ep{\normalfont{}\arabic*}}]
			\item\label{item:continuous} the map $\mathcal{C} \to \W_\G \colon (\mathfrak{a}, \mathfrak{b}) \mapsto \mathfrak{a} \times \mathfrak{b}$ is continuous;
			
			\item\label{item:bound_on_steps} for all $S \in \fins{\G}$, $k \in \N^+$, and $\epsilon > 0$, there is $N \in \N^+$ such that for all $(\mathfrak{a}, \mathfrak{b}) \in \mathcal{C}$,
			\[
			N_{S, k}(\mathfrak{a}, \mathfrak{b}, \epsilon) \leq N.
			\]
		\end{enumerate}
	\end{theo}
	\begin{proof}[\textsc{Proof}]
		We start with the implication \ref{item:continuous} $\Longrightarrow$ \ref{item:bound_on_steps}. Suppose that \ref{item:continuous} holds and assume that for some $S \in \fins{\G}$, $k \in \N^+$, and $\epsilon > 0$, there is a sequence of pairs $(\mathfrak{a}_n, \mathfrak{b}_n) \in \mathcal{C}$ with $N_{S, k} (\mathfrak{a}_n, \mathfrak{b}_n, \epsilon) \longrightarrow \infty$. Since $\mathcal{C}$ is compact, we may pass to a subsequence so that $(\mathfrak{a}_n, \mathfrak{b}_n) \longrightarrow (\mathfrak{a}, \mathfrak{b}) \in \mathcal{C}$. By \ref{item:continuous}, we then also have $\mathfrak{a}_n \times \mathfrak{b}_n \longrightarrow \mathfrak{a} \times \mathfrak{b}$. Set $N \defeq N_{S, k} (\mathfrak{a}, \mathfrak{b}, \epsilon/3)$.
		
		Let $\alpha_n \colon \G \acts (X_n, \mu_n)$, $\beta_n \colon \G \acts (Y_n, \nu_n)$, $\alpha \colon \G \acts (X, \mu)$, and $\beta \colon \G \acts (Y, \nu)$ be representatives of the weak equivalence classes $\mathfrak{a}_n$, $\mathfrak{b}_n$, $\mathfrak{a}$, and $\mathfrak{b}$ respectively. We claim that $N_{S, k}(\alpha_n, \beta_n, \epsilon) \leq N$ for all sufficiently large $n \in \N$, contradicting the choice of $(\mathfrak{a}_n, \mathfrak{b}_n)$. Indeed, take any $u \in \wecSet {\alpha_n \times \beta_n} S k$. If $n$ is large enough, then there is $v \in \wecSet {\alpha \times \beta} S k$ such that \[\dist_\infty(u, v) \,<\, \epsilon/3.\] By the choice of $N$, there is a step function $f \in \Step_{k, N}(X, \mu; Y, \nu)$ such that
		\[
			\dist_\infty(v, \wecSetFun {\alpha \times \beta} S k f) \,<\, \epsilon/3.
		\]
		Let $g \in \Meas_N(X, \mu)$, $h \in \Meas_N(Y, \nu)$, and $\phi \colon N \times N \to k$ be such that $f = \phi \circ (g, h)$. If $n$ is large enough, then there exist maps $\tilde{g} \in \Meas_N(X_n, \mu_n)$ and $\tilde{h} \in \Meas_N(Y_n, \nu_n)$ satisfying
		\[
			 \dist_\infty(\wecSetFun {\alpha_n} S N {\tilde{g}}, \wecSetFun \alpha S N g) + \dist_\infty(\wecSetFun {\beta_n} S N {\tilde{h}}, \wecSetFun \beta S N h) \,<\, \epsilon N^{-4}/3.
		\]
		Let $\tilde{f} \defeq \phi \circ (\tilde{g}, \tilde{h})$. From Propositions~\ref{prop:conv} and \ref{prop:conv_Lip}, it follows that
		\[
			\dist_\infty(u, \wecSetFun {\alpha_n \times \beta_n} S k {\tilde{f}}) \,<\, \epsilon/3 + \epsilon/3 + N^4 \cdot (\epsilon N^{-4}/3) \,=\, \epsilon,
		\]
		as desired.
		
		Now we proceed to the implication \ref{item:bound_on_steps} $\Longrightarrow$ \ref{item:continuous}. Suppose that \ref{item:bound_on_steps} holds and let $(\mathfrak{a}_n, \mathfrak{b}_n)$, $(\mathfrak{a}, \mathfrak{b}) \in \mathcal{C}$ be such that $(\mathfrak{a}_n, \mathfrak{b}_n) \longrightarrow (\mathfrak{a}, \mathfrak{b})$. We have to show that $\mathfrak{a}_n \times \mathfrak{b}_n \longrightarrow \mathfrak{a} \times \mathfrak{b}$. Let $\alpha_n \colon \G \acts (X_n, \mu_n)$, $\beta_n \colon \G \acts (Y_n, \nu_n)$, $\alpha \colon \G \acts (X, \mu)$, and $\beta \colon \G \acts (Y, \nu)$ be representatives of the weak equivalence classes $\mathfrak{a}_n$, $\mathfrak{b}_n$, $\mathfrak{a}$, and $\mathfrak{b}$ respectively. We must argue that for any $S \in \fins{\G}$, $k \in \N^+$, and $\epsilon > 0$ and for all sufficiently large $n \in \N$,
		\begin{align}\label{eq:lower}
			&\wecSet {\alpha \times \beta} S k \,\subseteq\, \Ball_\epsilon(\wecSet {\alpha_n \times \beta_n} S k);\\
			&\wecSet {\alpha_n \times \beta_n} S k \,\subseteq\, \Ball_\epsilon(\wecSet {\alpha \times \beta} S k).\label{eq:upper}
		\end{align}
		
		To prove \eqref{eq:lower}, let $N \defeq N_{S, k} (\alpha, \beta, \epsilon/2)$ and consider any $u \in \wecSet {\alpha \times \beta} S k$. By the choice of $N$, there is a step function $f \in \Step_{k, N}(X, \mu; Y, \nu)$ such that
		\[
			\dist_\infty(u, \wecSetFun {\alpha \times \beta} S k f) \,<\, \epsilon/2.
		\]
		Let $g \in \Meas_N(X, \mu)$, $h \in \Meas_N(Y, \nu)$, and $\phi \colon N \times N \to k$ be such that $f = \phi \circ (g, h)$. If $n$ is large enough, then there exist maps $\tilde{g} \in \Meas_N(X_n, \mu_n)$ and $\tilde{h} \in \Meas_N(Y_n, \nu_n)$ satisfying
		\[
		\dist_\infty(\wecSetFun {\alpha_n} S N {\tilde{g}}, \wecSetFun \alpha S N g) + \dist_\infty(\wecSetFun {\beta_n} S N {\tilde{h}}, \wecSetFun \beta S N h) \,<\, \epsilon N^{-4}/2.
		\]
		Let $\tilde{f} \defeq \phi \circ (\tilde{g}, \tilde{h})$. From Propositions~\ref{prop:conv} and \ref{prop:conv_Lip}, it follows that
		\[
			\dist_\infty(u, \wecSetFun {\alpha_n \times \beta_n} S k {\tilde{f}}) \,<\, \epsilon/2 + N^4 \cdot (\epsilon N^{-4}/2) \,=\, \epsilon,
		\]
		i.e., $u \in \Ball_\epsilon(\wecSet {\alpha_n \times \beta_n} S k)$, as desired. Notice that this argument did not involve assumption \ref{item:bound_on_steps}.
		
		To prove \eqref{eq:upper}, we use \ref{item:bound_on_steps} and choose $N$ so that $N_{S, k}(\alpha_n, \beta_n, \epsilon) \leq N$ for all $n \in \N$. Consider any $u \in \wecSet {\alpha_n \times \beta_n} S k$. Then there is a step function $f \in \Step_{k, N}(X_n, \mu_n; Y_n, \nu_n)$ such that
		\[
		\dist_\infty(u, \wecSetFun {\alpha_n \times \beta_n} S k f) \,<\, \epsilon/2.
		\]
		Let $g \in \Meas_N(X_n, \mu_n)$, $h \in \Meas_N(Y_n, \nu_n)$, and $\phi \colon N \times N \to k$ be such that $f = \phi \circ (g, h)$. If $n$ is large enough, then there exist maps $\tilde{g} \in \Meas_N(X, \mu)$ and $\tilde{h} \in \Meas_N(Y, \nu)$ satisfying
		\[
		\dist_\infty(\wecSetFun {\alpha} S N {\tilde{g}}, \wecSetFun {\alpha_n} S N g) + \dist_\infty(\wecSetFun {\beta} S N {\tilde{h}}, \wecSetFun {\beta_n} S N h) \,<\, \epsilon N^{-4}/2.
		\]
		Let $\tilde{f} \defeq \phi \circ (\tilde{g}, \tilde{h})$. From Propositions~\ref{prop:conv} and \ref{prop:conv_Lip}, it follows that
		\[
		\dist_\infty(u, \wecSetFun {\alpha \times \beta} S k {\tilde{f}}) \,<\, \epsilon/2 + N^4 \cdot (\epsilon N^{-4}/2) \,=\, \epsilon,
		\]
		i.e., $u \in \Ball_\epsilon(\wecSet {\alpha \times \beta} S k)$, and we are done.
	\end{proof}
	
	\section{Proof of Theorem~\ref{theo:main}}\label{sec:proof}
	
	\subsection{Expansion in $\SL_d(\Z/n\Z)$}
	
	For $n \in \N^+$, we use $\pi_n$ to indicate reduction modulo $n$ in various contexts. That is, we slightly abuse notation and give the same name to the residue maps
	\[
		\pi_n \colon \Z\to \Z/n\Z, \qquad \pi_n \colon \SL_d(\Z) \to \SL_d(\Z/n\Z), \qquad  \text{etc}.
	\]
	
	Let $G$ be a nontrivial finite group. For $A$, $S \subseteq G$, the \emphd{boundary}\footnote{For our purposes it will be more convenient to consider the vertex rather than the edge boundary.} of $A$ with respect to $S$ is
	\[
		\partial (A, S) \defeq \set{a \in A \,:\, Sa \not \subseteq A}.
	\]
	The \emphd{Cheeger constant} $h(G, S)$ of $G$ with respect to $S$ is given by
	\[
		h(G, S) \defeq \min_A \frac{|\partial(A, S)|}{|A|},
	\]
	where the minimum is taken over all nonempty subsets $A \subseteq G$ of size at most $|G|/2$. Notice that we have $h(G, S) > 0$ if and only if $S$ generates $G$. Indeed, let $\langle S \rangle$ be the subgroup of $G$ generated by $S$. If $\langle S \rangle \neq G$, then $|\langle S \rangle| \leq |G|/2$, while $\partial(\langle S \rangle, S) = \0$, hence $h(G, S) = 0$. Conversely, if $h(G, S) = 0$, then there is a nonempty proper subset $A \varsubsetneq G$ closed under left multiplication by the elements of $S$. This means that $A$ is a union of right cosets of $\langle S \rangle$, and thus $\langle S \rangle \neq G$.
	
	\begin{theo}[{Bourgain--Varj\'u \cite[Theorem~1]{BV}}]\label{theo:BV}
		Let $d \geq 2$ and let $S \in \fins{\SL_d(\Z)}$ be a  finite symmetric subset such that the subgroup $\langle S \rangle$ of $\SL_d(\Z)$ generated by $S$ is Zariski dense in $\SL_d(\R)$. Then there exist $n_0 \in \N^+$ and $\epsilon > 0$ such that for all $n \geq 2$, if $\operatorname{gcd}(n, n_0) = 1$, then
		\[
			h(\SL_d(\Z/n\Z), \pi_n(S)) \geq \epsilon.
		\]
	\end{theo}
	
	Theorem~\ref{theo:BV} is an outcome of a long series of contributions by a number of researchers; for more background, see \cite{BV, Tao_book} and the references therein.
	
	A finite group $G$ is called \emphd{$D$-quasirandom}, where $D \geq 1$, if every nontrivial unitary representation of $G$ has dimension at least $D$ (a representation $\rho$ of $G$ is nontrivial if $\rho(a) \neq 1$ for some $a \in G$). This notion was introduced by Gowers \cite{Gowers}. For a map $\zeta \colon G \to \mathbb{C}$, we write
	\[
		\mathbf{E} \zeta \defeq \frac{1}{|G|}\sum_{x \in G} \zeta(x), \qquad \|\zeta\|_\infty \defeq \max_{x \in G} |\zeta(x)|, \qquad \text{and} \qquad \|\zeta\|_2 \defeq \sqrt{\sum_{x \in G} |\zeta(x)|^2}.
	\]
	Given $\zeta$, $\eta \colon G \to \mathbb{C}$, define the \emphd{convolution} $\zeta \ast \eta \colon G \to \mathbb{C}$ of $\zeta$ and $\eta$ by the formula
	\[
		(\zeta \ast \eta)(x) \defeq \sum_{ab \,=\, x} \zeta(a)\eta(b), 
	\]
	where the sum is taken over all pairs of $a$, $b \in G$ such that $ab=x$.
	
	\begin{theo}[{\cite[Proposition 1.3.7]{Tao_book}}]\label{theo:weak_mixing}
		Let $G$ be a finite group and let $\zeta$, $\eta \colon G \to \mathbb{C}$. Suppose that $G$ is $D$-quasirandom. If $\mathbf{E}\zeta = \mathbf{E}\eta = 0$, then
		\[
			\|\zeta \ast \eta\|_2 \leq \sqrt{\frac{|G|}{D}}\|\zeta\|_2 \|\eta\|_2.
		\]
	\end{theo}

	In order to apply Theorem~\ref{theo:weak_mixing}, we will need the following variation of Frobenius's lemma:
	
	\begin{prop}[{cf. \cite[Lemma 1.3.3]{Tao_book}}]\label{prop:Frobenius}
		Let $d$, $n \geq 2$ and let $p$ be the smallest prime divisor of $n$. Then the group $\SL_d(\Z/n\Z)$ is $(p-1)/2$-quasirandom.
	\end{prop}
	\begin{proof}[\textsc{Proof}]
		The statement is trivial for $p = 2$, so assume that $p$ is odd. Write $n$ as a product of powers of distinct primes: $n = p_1^{k_1} \cdots p_r^{k_r}$. Then, by the Chinese remainder theorem,
		\[
			\SL_d(\Z/n\Z) \,\cong\, \SL_d(\Z/p_1^{k_1}\Z) \times \cdots \times \SL_d(\Z/p_r^{k_r}\Z).
		\]
		Since the product of $D$-quasirandom groups is again $D$-quasirandom \cite[Exercise 1.3.2]{Tao_book}, it is enough to consider the case when $r = 1$ and $n = p^k$.
		
		Let $\rho$ be a nontrivial finite-dimensional unitary representation of $\SL_d(\Z/p^k\Z)$. By \cite[Theorem~4.3.9]{H-OM}, the group $\SL_d(\Z/p^k\Z)$ is generated by the \emphd{elementary} matrices, i.e., those that differ from the identity matrix in precisely one off-diagonal entry. Thus, there exists an elementary matrix $e \in \SL_d(\Z/p^k\Z)$ such that $\rho(e) \neq 1$. Without loss of generality, we may assume that $e$ is of the form
		\[
			e = \left(\begin{array}{cccc}
				1 & a & \cdots & 0 \\
				0 & 1 & \cdots & 0 \\
				\vdots & \vdots & \ddots & \vdots \\
				0 & 0 & \cdots & 1
			\end{array}\right),
		\]
		where $0 \neq a \in \Z/p^k\Z$. Choose $e$ so as to maximize the power of $p$ that divides $a$. Let $\lambda$ be an arbitrary eigenvalue of $\rho(e)$ not equal to $1$ (such $\lambda$ exists since $\rho(e) \neq 1$ and is unitary). We have
		\[
			e^p = \left(\begin{array}{cccc}
			1 & pa & \cdots & 0 \\
			0 & 1 & \cdots & 0 \\
			\vdots & \vdots & \ddots & \vdots \\
			0 & 0 & \cdots & 1
			\end{array}\right),
		\]
		so, by the choice of $e$, $\rho(e)^p = \rho(e^p) = 1$. Hence, $\lambda^p = 1$, so the values $\lambda$, $\lambda^2$, \ldots, $\lambda^{p-1}$ are pairwise distinct. Let $b \in \N^+$ be an integer coprime to $p$ and let $c \defeq b^2$. Since $b$ is  invertible in $\Z/p^k\Z$, we can form a diagonal matrix $h \in \SL_d(\Z/p^k\Z)$ with entries $(b, b^{-1}, 1 \ldots, 1)$. Then $h^{-1}$ is the diagonal matrix with entries $(b^{-1}, b, 1, \ldots, 1)$, and we have
		\[
		heh^{-1} = \left(\begin{array}{cccc}
			1 & ca & \cdots & 0 \\
			0 & 1 & \cdots & 0 \\
			\vdots & \vdots & \ddots & \vdots \\
			0 & 0 & \cdots & 1
			\end{array}\right) = e^c.
		\]
		This shows that $e$ and $e^c$ are conjugate in $\SL_d(\Z/p^k\Z)$, and hence $\rho(e)$ and $\rho(e)^c$ are conjugate as well. Since $\lambda^c$ is an eigenvalue of $\rho(e)^c$, it must also be an eigenvalue of $\rho(e)$. It remains to notice that there exist $(p-1)/2$ choices for $c$ that are distinct modulo $p$ (corresponding to the $(p-1)/2$ nonzero quadratic residues modulo $p$), so $\rho(e)$ must have at least $(p-1)/2$ distinct eigenvalues, which is only possible if the dimension of $\rho$ is at least $(p-1)/2$.
	\end{proof}
	
	\subsection{The main lemma}
	
	For the rest of Section~\ref{sec:proof}, fix $d \geq 2$ and let $\G$ be a subgroup of $\SL_d(\Z)$ that is Zariski dense in $\SL_d(\R)$.
	
	For $n \geq 2$, define $G_n \defeq \SL_d(\Z/n\Z)$. Let $\alpha_n \colon \G \acts G_n$ be the action given by
	\[
		\gamma \cdot x \defeq \pi_n(\gamma)x \qquad \text{for all } \gamma \in \G \text{ and } x \in G_n.
	\]
	We view $\alpha_n$ as a \pmp action by equipping $G_n$ with the uniform probability measure (to simplify notation, we will avoid mentioning this measure explicitly).
	
	The group $\G$ has a Zariski dense finitely generated subgroup (by Tits's theorem~\cite[Theorem~3]{Tits}, such a subgroup can be chosen to be free of rank $2$), so fix an arbitrary finite symmetric set $S \in \fins{\G}$ such that the group $\langle S \rangle$ is Zariski dense in $\SL_d(\R)$. Fix $n_0 \in \N^+$ and $\epsilon >0$ provided by Theorem~\ref{theo:BV} applied to $S$ and let \[\delta \defeq \frac{\epsilon}{32|S|}.\]
	Define $\uhalf \in [0;1]^{S \times 2 \times 2}$ by setting, for all $\gamma \in S$ and $i$, $j <2$,
	\[
		\uhalf(\gamma, i, j) \defeq \begin{cases}
			1/2 &\text{if } i = j;\\
			0 &\text{if } i \neq j.
		\end{cases}
	\]
	The heart of the proof of Theorem~\ref{theo:main} lies in the following lemma:
	
	\begin{lemma}\label{lemma:finite}
		Let $n$, $m \geq 2$ be such that $n$ divides $m$ and $\operatorname{gcd}(m, n_0) = 1$. Let $p$ be the smallest prime divisor of $n$ and let
		\[
			N \defeq \left\lfloor \frac{1}{25}\sqrt{p-1}\right\rfloor.
		\]
		Assume that $N \geq 1$. Then $\uhalf \in \wecSet {\alpha_n \times \alpha_m} S 2$, yet for all $f \in \Step_{2, N} (G_n; G_m)$, we have
			\[
				\dist_\infty(\uhalf, \wecSetFun {\alpha_n \times \alpha_m} S 2 f) \geq \delta.
			\]
		In particular, $N_{S, 2}(\alpha_n, \alpha_m, \delta) > N$. 
	\end{lemma}
	\begin{proof}[\textsc{Proof}]
		Let $\proj_2 \colon G_n \times G_m \to G_m$ denote the projection on the second coordinate. 
		Note that, by definition, the map $\proj_2$ is equivariant. Since $n$ divides $m$, there is a well-defined reduction modulo $n$ map $\pi_n \colon G_m \to G_n$, and it is surjective. For $z \in G_n$, define
		\[
			\mathcal{O}_z \defeq \set{(x, y) \in G_n \times G_m \,:\, x = \pi_n(y)z}.
		\]
		Evidently, the set $\mathcal{O}_z$ is $(\alpha_n \times \alpha_m)$-invariant. Furthermore, the map $\proj_2$ establishes an equivariant bijection between $\mathcal{O}_z$ and $G_m$. Since $\operatorname{gcd}(m,n_0) = 1$, Theorem~\ref{theo:BV} implies that the action $\alpha_m$ is transitive, and hence so is the restriction of the action $\alpha_n \times \alpha_m$ to $\mathcal{O}_z$. Therefore, the orbits of $\alpha_n \times \alpha_m$ are precisely the sets $\mathcal{O}_z$ for $z\in G_n$.
		
		Given a subset $Z \subseteq G_n$, define $f_Z \colon G_n \times G_m \to 2$ by
		\[
			f_Z(x, y) \defeq \begin{cases}
			0 &\text{if } \pi_n(y)^{-1}x \not\in Z;\\
			1 &\text{if } \pi_n(y)^{-1}x \in Z.
			\end{cases}
		\]
		The functions of the form $f_Z$ for $Z \subseteq G_n$ are precisely the $(\alpha_n \times \alpha_m)$-invariant maps $G_n \times G_m \to 2$.
		
		Now we can show that $\uhalf \in \wecSet {\alpha_n \times \alpha_m} S 2$. The group $G_n$ contains an element of order $2$, namely the diagonal matrix with entries $(-1,-1,1,\ldots, 1)$, so $|G_n|$ is even. Hence, for any set $Z \subset G_n$ of size exactly $|G_n|/2$, we have $\wecSetFun {\alpha_n \times \alpha_m} S 2 {f_Z} = \uhalf$, as desired.
		
		For $z \in G_n$ and $A \subseteq \mathcal{O}_z$, define the \emphd{boundary} of $A$ by
		\[
		\partial A \defeq \set{(x,y) \in A \,:\, S \cdot (x,y) \not \subseteq A}.
		\]
		Suppose that $|A| \leq |G_m|/2$ (note that $|G_m| = |\mathcal{O}_z|$). Then, since $\proj_2$ establishes an equivariant bijection between $\mathcal{O}_z$ and $G_m$, Theorem~\ref{theo:BV} yields
		\begin{equation}\label{eq:boundary_in_orbit}
		|\partial A| \geq \epsilon |A|.
		\end{equation}
		
		\begin{smallclaim}\label{claim:invariant}
			Let $f \colon G_n \times G_m \to 2$ be such that
			\begin{equation}\label{eq:almost_inv}
				\dist_\infty(\uhalf, \wecSetFun {\alpha_n \times \alpha_m} S 2 f) < \delta.
			\end{equation}
			Then there is a set $Z \subseteq G_n$ such that $\dist(f, f_Z) < 1/16$.
		\end{smallclaim}
		\begin{claimproof}
			For each $\gamma \in S$, let
			\[
				B_\gamma \defeq \set{(x,y) \in G_n \times G_m \,:\, f(x,y) \neq f(\gamma \cdot x, \gamma \cdot y)},
			\]
			and define $B \defeq \bigcup_{\gamma \in S} B_\gamma$. By \eqref{eq:almost_inv}, for any $\gamma \in S$, we have
			\[
				\frac{|B_\gamma|}{|G_n| |G_m|} \,=\, \wecSetFun {\alpha_n \times \alpha_m} S 2 f (\gamma, 0, 1) + \wecSetFun {\alpha_n \times \alpha_m} S 2 f (\gamma, 1, 0) \,<\, 2\delta,
			\]
			and therefore \[|B| \,<\, 2\delta |S| |G_n| |G_m| \,=\, \frac{\epsilon}{16}|G_n||G_m|.\] We will show that the set
			\[
				Z \defeq \set{z \in G_n \,:\, f(x, y) = 1 \text{ for at least } |G_m|/2 \text{ pairs } (x,y) \in \mathcal{O}_z}
			\]
			is as desired. Define
			\[
				A \defeq \set{(x, y) \in G_n \times G_m \,:\, f(x,y) \neq f_Z(x,y)},
			\]
			so $\dist (f, f_Z)  = |A|/(|G_n||G_m|)$. Take any $z \in G_n$. By the definition of $Z$, we have $|A \cap \mathcal{O}_z| \leq |G_m|/2$, and hence, by \eqref{eq:boundary_in_orbit},
			\[
				|\partial(A \cap \mathcal{O}_z)| \geq \epsilon |A \cap \mathcal{O}_z|.
			\]
			Note that $\partial (A \cap \mathcal{O}_z) \subseteq B \cap \mathcal{O}_z$, so we have
			\[
				|B \cap \mathcal{O}_z| \geq |\partial (A \cap \mathcal{O}_z)| \geq \epsilon|A \cap \mathcal{O}_z|.
			\]
			Hence,
			\[
				|A| \,=\, \sum_{z \in G_n} |A \cap \mathcal{O}_z| \,\leq\, \sum_{z \in G_n} \epsilon^{-1} |B \cap \mathcal{O}_z| \,=\, \epsilon^{-1} |B| \,<\, \frac{1}{16}|G_n||G_m|.
			\]
			In other words, $\dist(f, f_Z) < 1/16$, as claimed.
		\end{claimproof}
		
		For $\zeta \colon G_n \to \mathbb{C}$ and $\xi \colon G_m \to \mathbb{C}$, define $\zeta \circledast \xi \colon G_n \to \mathbb{C}$ by the formula
		\[
			(\zeta \circledast \xi)(x) \defeq \sum_{a\pi_n(b) \,=\, x} \zeta(a) \xi(b), 
		\]
		where the sum is taken over all pairs of $a \in G_n$ and $b \in G_m$ such that $a\pi_n(b)=x$. We will need the following corollary of Theorem~\ref{theo:weak_mixing} and Proposition~\ref{prop:Frobenius}:
		
		\begin{smallclaim}\label{claim:triple_convolution}
			Let $\zeta$, $\eta \colon G_n \to \mathbb{C}$ and $\xi \colon G_m \to \mathbb{C}$. Then
			\[
				\|(\zeta \ast \eta) \circledast \xi \,-\, (\mathbf{E}\zeta) (\mathbf{E}\eta) (\mathbf{E} \xi) |G_n||G_m|\|_\infty \,\leq\, \sqrt{\frac{2|G_m|}{p-1}} \|\zeta\|_2 \|\eta\|_2 \|\xi\|_2.
			\]
		\end{smallclaim}
		\begin{claimproof}
			This is a variant of \cite[Exercise 1.3.12]{Tao_book}. After subtracting its expectation from each function, we may assume that $\mathbf{E}\zeta = \mathbf{E}\eta = \mathbf{E}\xi = 0$. 
			By the Cauchy--Schwarz inequality, we have
			\[
				\|(\zeta \ast \eta) \circledast \xi \|_\infty 
				\,\leq\, \sqrt{\frac{|G_m|}{|G_n|}}\|\zeta \ast \eta\|_2 \|\xi\|_2,
			\]
			while Theorem \ref{theo:weak_mixing} and Proposition~\ref{prop:Frobenius} yield
			\[
			\|\zeta \ast \eta\|_2
			\,\leq\, \sqrt{\frac{2|G_n|}{p-1}} \|\zeta\|_2 \|\eta\|_2. \qedhere
			\]
		\end{claimproof}
		
		We use Claim~\ref{claim:triple_convolution} to prove that invariant maps are hard to approximate by step functions:
		
		\begin{smallclaim}\label{claim:not_invariant}
			Let $f \in \Step_{2, N} (G_n; G_m)$ and $Z \subseteq G_n$. Suppose that
			$\min \set{|Z|, |G_n| - |Z|} \geq |G_n|/4$.
			Then $\dist (f, f_Z) \geq 1/8$. 
		\end{smallclaim}
		\begin{claimproof}
			Let $g \colon G_n \to N$, $h \colon G_m \to N$, and $\phi \colon N \times N \to 2$ be such that $f = \phi \circ (g,h)$. For $i < N$, set \[X_i \defeq g^{-1}(i).\] Thus, $\set{X_i \,:\, i < N}$ is a partition of $G_n$ into $N$ pieces. Given $i < N$ and $j < 2$, let
			\[
				Y_{i,j} \,\defeq\, \set{y \in G_m \,:\, \phi(i, h(y)) = j} \,=\, \set{y \in G_m \,:\, f(x, y) = j \text{ for all } x \in X_i}.
			\]
			Note that $Y_{i,0} \cup Y_{i,1} = G_m$. Define
			\[
				A \defeq \set{(x, y) \in G_n \times G_m \,:\, f(x,y) \neq f_Z(x,y)},
			\]
			so $\dist(f, f_Z) = |A|/(|G_n||G_m|)$. Let $\mathbf{1}_{G_n}$ be the identity element of $G_n$, and for each set $F \subseteq G_n$, let $\mathbbm{1}_F \colon G_n \to 2$ denote the indicator function of $F$. Then, for any $i < N$, we have
			\begin{align*}
				|(X_i \times Y_{i,0}) \cap A| \,&=\, |\set{(x, y) \in X_i \times Y_{i,0} \,:\, f_Z(x,y) = 1}| \\
				&=\, |\set{(z, x, y) \in Z \times X_i \times Y_{i,0} \,:\, z x^{-1} \pi_n(y) = \mathbf{1}_{G_n} }|\\
				&=\, ((\mathbbm{1}_Z \ast \mathbbm{1}_{X^{-1}_i}) \circledast \mathbbm{1}_{Y_{i,0}})(\mathbf{1}_{G_n}).
			\end{align*}
			By Claim~\ref{claim:triple_convolution}, the last expression is at least 
			\[
				\frac{|Z||X_i||Y_{i,0}|}{|G_n|} \,-\, \sqrt{\frac{2|G_m||Z||X_i||Y_{i,0}|}{p-1}} \,\geq\, \frac{|X_i||Y_{i,0}|}{4} \,-\, \sqrt{\frac{2}{p-1}} |G_n||G_m|.
			\]
			Similarly, we have
			\[
				|(X_i \times Y_{i,1}) \cap A| \,\geq\, \frac{|X_i||Y_{i,1}|}{4} \,-\, \sqrt{\frac{2}{p-1}} |G_n||G_m|,
			\]
			and hence
			\[
				|(X_i \times G_m) \cap A| \,\geq\, \frac{|X_i||G_m|}{4} \,-\, \sqrt{\frac{8}{p-1}} |G_n||G_m|.
			\]
			Therefore,
			\[
				|A| \,=\, \sum_{i < N} |(X_i \times G_m) \cap A|
				\,\geq\, \left(\frac{1}{4} - N\sqrt{\frac{8}{p-1}}\right)|G_n||G_m| \,>\, \frac{1}{8}|G_n||G_m|,
			\]
			and thus $\dist(f, f_Z) > 1/8$, as desired.
		\end{claimproof}
		
		It remains to combine Claims~\ref{claim:invariant} and \ref{claim:not_invariant}. Suppose that $f \in \Step_{2,N}(G_n;G_m)$ satisfies
		\[
		\dist_\infty(\uhalf, \wecSetFun {\alpha_n \times \alpha_m} S 2 f) < \delta.
		\]
		By Claim~\ref{claim:invariant}, there is a set $Z \subseteq G_n$ such that $\dist(f, f_Z) <1/16$. By Proposition~\ref{prop:Lipschitz}, we have
		\[
			\dist_\infty(\uhalf, \wecSetFun {\alpha_n \times \alpha_m} S 2 {f_Z}) \,<\, \delta + 2 \cdot (1/16) \,<\,1/4.
		\]
		In particular, for any $\gamma \in S$,
		\[
			\frac{|Z|}{|G_n|} \,=\, \wecSetFun {\alpha_n \times \alpha_m} S 2 {f_Z} (\gamma, 1, 1) \,>\, \uhalf(\gamma, 1, 1) - 1/4 \,=\, 1/4,
		\]
		i.e., $|Z| \geq |G_n|/4$, and, similarly, $|G_n| - |Z| \geq |G_n|/4$. Therefore, by Claim~\ref{claim:not_invariant}, $\dist(f, f_Z) \geq 1/8$, which is a contradiction. The proof of Lemma~\ref{lemma:finite} is complete.
	\end{proof}
	
	\subsection{Finishing the proof}
	
	We say that $\mathcal{N} \subseteq \N^+$ is a \emphd{directed set} if $\mathcal{N}$ is infinite and for any two elements $n_1$, $n_2 \in \mathcal{N}$, there is some $m \in \mathcal{N}$ divisible by both $n_1$ and $n_2$. Each directed set $\mathcal{N} \subseteq \N^+$ gives rise to an inverse system consisting of the groups $(G_n)_{n \in \mathcal{N}}$ together with the homomorphisms $\pi_n \colon G_m \to G_n$ for every pair of $n$, $m \in \mathcal{N}$ such that $n$ divides $m$. The inverse limit of this system is an infinite profinite group, which we denote by $G_\mathcal{N}$. For example, if we let
	\[
		\mathcal{N}(p) \defeq \set{p, p^2, p^3, \ldots}
	\]
	for some prime $p$, then $G_{\mathcal{N}(p)}\cong\SL_d(\Z_p)$, where $\Z_p$ is the ring of $p$-adic integers.
	
	If $\mathcal{N} \subseteq \N^+$ is a directed set, then $\SL_d(\Z)$ naturally embeds into $G_\mathcal{N}$, so we can identify $\G$ with a subgroup of $G_\mathcal{N}$. This allows us to consider the left multiplication action $\alpha_\mathcal{N} \colon \G \acts G_\mathcal{N}$. As the group $G_\mathcal{N}$ is compact, we can equip $G_\mathcal{N}$ with the Haar probability measure and view $\alpha_\mathcal{N}$ as a \pmp action. Clearly, the action $\alpha_\mathcal{N}$ is free. Note that for each $n \in \mathcal{N}$, there is a well-defined reduction modulo $n$ map $\pi_n \colon G_\mathcal{N} \to G_n$, which is equivariant and pushes the Haar measure on $G_\mathcal{N}$ forward to the uniform probability measure on $G_n$. In particular, $\alpha_n$ is a factor of $\alpha_\mathcal{N}$, and hence $\alpha_n \preceq \alpha_\mathcal{N}$.
	
	The following is a direct consequence of Lemma~\ref{lemma:finite}:
	
	\begin{lemma}\label{lemma:infinite}
		Let $\mathcal{N}$, $\mathcal{M} \subseteq \N^+$ be directed sets such that $\mathcal{N} \subseteq \mathcal{M}$ and $\operatorname{gcd}(m, n_0) = 1$ for all $m \in \mathcal{M}$. Let $p$ be the smallest prime number that divides an element of $\mathcal{N}$ and let
		\[
		N \defeq \left\lfloor \frac{1}{25}\sqrt{p-1}\right\rfloor.
		\]
		Assume that $N \geq 1$. Then $\uhalf \in \wecSet {\alpha_\mathcal{N} \times \alpha_\mathcal{M}} S 2$, yet for all $f \in \Step_{2, N} (G_\mathcal{N}; G_\mathcal{M})$, we have
		\[
		\dist_\infty(\uhalf, \wecSetFun {\alpha_\mathcal{N} \times \alpha_\mathcal{M}} S 2 f) \geq \delta.
		\]
		In particular, $N_{S, 2}(\alpha_\mathcal{N}, \alpha_\mathcal{M}, \delta) > N$. 
	\end{lemma}
	\begin{proof}[\textsc{Proof}]
		To prove that $\uhalf \in \wecSet {\alpha_\mathcal{N} \times \alpha_\mathcal{M}} S 2$, take any $n \in \mathcal{N}$, $n \geq 2$. Since $\alpha_n \preceq \alpha_\mathcal{N}$, $\alpha_\mathcal{M}$, it follows from Corollary~\ref{corl:mult} that $\alpha_n \times \alpha_n \preceq \alpha_\mathcal{N} \times \alpha_\mathcal{M}$, and, by Lemma~\ref{lemma:finite}, we obtain \[\uhalf \in \wecSet {\alpha_n \times \alpha_n} S 2 \subseteq \wecSet {\alpha_\mathcal{N} \times \alpha_\mathcal{M}} S 2.\]
		Now suppose that some $f \in \Step_{2, N} (G_\mathcal{N}; G_\mathcal{M})$ satisfies
		\begin{equation}\label{eq:too_close}
		\dist_\infty(\uhalf, \wecSetFun {\alpha_\mathcal{N} \times \alpha_\mathcal{M}} S 2 f) < \delta.
		\end{equation}
		Let $g \in \Meas_N(G_\mathcal{N})$, $h \in \Meas_N(G_\mathcal{M})$, and $\phi \colon N \times N \to 2$ be such that $f = \phi \circ (g, h)$. After modifying the maps $g$ and $h$ on sets of arbitrarily small measure, we can arrange that there exist integers $n \in \mathcal{N}$ and $m \in \mathcal{M}$ and functions $\tilde{g} \colon G_n \to N$ and $\tilde{h} \colon G_m \to N$ such that
		\[
			g = \tilde{g} \circ \pi_n \qquad \text{and} \qquad h = \tilde{h} \circ \pi_m.
		\] 
		Using Propositions~\ref{prop:conv} and \ref{prop:conv_Lip}, we can ensure that inequality \eqref{eq:too_close} is still valid after this modification. We may furthermore assume that $n \geq 2$ and $n$ divides $m$ (the last part uses that $\mathcal{N} \subseteq \mathcal{M}$ and $\mathcal{M}$ is a directed set). Let $\tilde{f} \defeq \phi \circ (\tilde{g}, \tilde{h})$. Then $\tilde{f} \in \Step_{2,N}(G_n;G_m)$ and \[\wecSetFun {\alpha_n \times \alpha_m} S 2 {\tilde{f}} = \wecSetFun {\alpha_\mathcal{N} \times \alpha_\mathcal{M}} S 2 f.\]
		But the existence of such $\tilde{f}$ contradicts Lemma~\ref{lemma:finite}.
	\end{proof}
	
	Now we can complete the proof of Theorem~\ref{theo:main}:
	
	\begin{proof}[\textsc{Proof of Theorem~\ref{theo:main}}]
		Recall that $\G$ is a Zariski dense subgroup of $\SL_d(\Z)$ with $d \geq 2$; $S$ is a finite symmetric subset of $\G$ such that the group $\langle S \rangle$ is still Zariski dense; $n_0 \in \N^+$ and $\epsilon > 0$ are given by Theorem~\ref{theo:BV} applied to $S$; and $\delta = \epsilon/(32|S|)$.
		
		\ref{item:squaring} By Lemma~\ref{lemma:infinite}, we have
		\[
			\lim_{p\text{ prime}}N_{S, 2}(\alpha_{\mathcal{N}(p)}, \alpha_{\mathcal{N}(p)}, \delta) \,=\, \infty.
		\]
		The desired conclusion follows by applying Theorem~\ref{theo:criterion} to the set $\mathcal{C}\defeq \set{(\mathfrak{a}, \mathfrak{a}) \,:\, \mathfrak{a} \in \WFree_\G}$.
		
		\ref{item:fiber} Let $\mathcal{M} \defeq \set{m \in \N^+ \,:\, \operatorname{gcd}(m, n_0) = 1}$. Then $\mathcal{M}$ is a directed set, and we claim that $\mathfrak{b} \defeq \wec{\alpha_\mathcal{M}}$ is as desired. Indeed, by Lemma~\ref{lemma:infinite}, we have
		\[
			\lim_{p\text{ prime}}N_{S, 2}(\alpha_{\mathcal{N}(p)}, \alpha_{\mathcal{M}}, \delta) \,=\, \infty,
		\]
		so it remains to apply Theorem~\ref{theo:criterion} to the set $\mathcal{C}\defeq \set{(\mathfrak{a}, \mathfrak{b}) \,:\, \mathfrak{a} \in \WFree_\G}$.
	\end{proof}

		\printbibliography

\end{document}